\documentclass[12pt,a4paper]{article}
\pdfoutput=1
\usepackage[utf8]{inputenc}
\usepackage{amsmath}
\usepackage[english]{babel}
\usepackage{amsfonts}
\usepackage{amssymb}
\usepackage{graphicx}
\usepackage{url}
\usepackage{amssymb,amsmath,mathtools,proof,amsthm}
\usepackage{makeidx}
\usepackage{lscape}
\usepackage{latexsym}
\usepackage{hyperref}
\usepackage[left=3cm, right=3cm, bottom=3cm, top=3cm]{geometry}
\usepackage[dvipsnames]{xcolor}
\usepackage[protrusion=true,expansion=true]{microtype}
\usepackage{hyperref}
\usepackage{orcidlink}

\theoremstyle{definition}
\newtheorem{definition}{Definition}
\newtheorem{theorem}{Theorem}
\newtheorem{lemma}{Lemma}

\newcommand{\NN}{\mathbb{N}}
\newcommand{\ZZ}{\mathbb{Z}}
\newcommand{\QQ}{\mathbb{Q}}
\newcommand{\RR}{\mathbb{R}}

\newcommand{\ZX}{\mathbb{Z}[X]}

\newcommand{\Neg}{%
  \mathord{\hbox{\raisebox{0.15ex}{$\neg$}\hspace{-0.67em}\raisebox{-0.15ex}{$\neg$}}}
}

\title{\huge Material Interpretation and Constructive Analysis of Maximal Ideals in $\ZZ[X]$\footnote{This work was founded by the Austrian Science Fund - FWF under Grant 10.55776/ESP576.\\ Great thanks also go to the anonymous reviewers, whose comments and suggestions have substantially enhanced the quality of this article.}}

\author{
Franziskus Wiesnet \orcidlink{0000-0003-3870-6984}
\\
{\normalsize Vienna University of Technology}\\
{\normalsize \href{mailto:franziskus.wiesnet@tuwien.ac.at}{franziskus.wiesnet@tuwien.ac.at}}
}
\date{}
\begin{document}
\maketitle
\begin{abstract}
This article presents the concept of material interpretation as a method to transform classical proofs into constructive ones. Using the case study of maximal ideals in \texorpdfstring{$\ZZ[X]$}{Z[X]}, it demonstrates how a classical implication \texorpdfstring{$A \to B$}{``A implies B"} can be rephrased as a constructive disjunction \texorpdfstring{$\Neg A \vee B$}{``-A or B"}, with \texorpdfstring{$\Neg A$}{-A} representing a strong form of negation. The approach is based on Gödel’s Dialectica interpretation, the strong negation, and potentially Herbrand disjunctions.

The classical proof that every maximal ideal in \texorpdfstring{$\ZZ[X]$}{Z[X]} contains a prime number is revisited, highlighting its reliance on non-constructive principles such as the law of excluded middle. A constructive proof is then developed, replacing abstract constructs with explicit case distinctions and direct computations in \texorpdfstring{$\ZZ[X]$}{Z[X]}. This proof clarifies the logical structure and reveals computational content.

The article discusses broader applications, such as Zariski’s Lemma, Hilbert’s Nullstellensatz, and the Universal Krull-Lindenbaum Lemma, with an emphasis on practical implementation using  tools such as Python and proof assistants. The material interpretation offers a promising framework for bridging classical and constructive mathematics, enabling algorithmic implementations.

\end{abstract}
\section{Motivation: Material Interpretation}
\label{Sec:Intro}
Based on the case study of maximal ideals in $\mathbb{Z}[X]$, our aim is to illustrate the concept of material interpretation. The approach begins with a (possibly classical) proof of a statement of the form $A \to B$ and aims to transform it as far as possible into a proof of the form $\Neg A \vee B$. Here, $\Neg A$ should represent a stronger form of the negation of $A$ which is quite similar to what is known as \textit{strong negation} \cite{gurevich1977intuitionistic,koepp2024strong,nelson1949constructible}.
For our purposes, it is sufficient to know that the strong negation of $A$ is a statement with constructive content that implies the ordinary negation $\neg A$ and in particular
\begin{align}
\label{Form:StrongNeg}
\Neg(\forall_x(A_1\wedge\dots \wedge A_k))\quad \Leftrightarrow \quad \exists_x (\Neg A_1 \vee\dots\vee \Neg A_k). 
\end{align}
In our case, the strong negation for prime formulas is simply the ordinary negation.
Many algebraic definitions, in particular that of a maximal ideal, have exactly the form $\forall_x(A_1\wedge\dots \wedge A_k)$. Therefore, we have a way to redefine their negation in a constructive stronger version.

A formula of the form $\Neg A \vee B$ can often be found using Gödel's Dialectica interpretation \cite{goedel1958ueber}, as Gödel’s interpretation was used in \cite{powell2019algorithmic,powell2022universal} to motivate a similar constructive approach. 
For example, if in Formula (\ref{Form:StrongNeg}) the $A_i$ are atomic and decidable formulas and $\Neg$
denotes the standard negation, then the right-hand side is equivalent to Gödel's Dialectica interpretation of the left-hand side.
Herbrand disjunctions could also play a role in this context, as they have a similar form and can be derived from the Dialectica interpretation as well \cite{gerhardy2005extracting}.

Since a classical proof of $A\to B$ typically relies on some non-constructive axioms, it will generally not be possible to transform it directly into a constructive proof. Therefore, the idea is to slightly modify $A$ and $B$ to $\tilde{A}$ and $\tilde{B}$ in advance to obtain a constructive proof of $\Neg \tilde{A} \vee \tilde{B}$. The new statements $\tilde{A}$ and $\tilde{B}$ should remain as close as possible to $A$ and $B$ and the modifications should be derivable from the classical proof. At first glance, this may seem very abstract, but let us now examine it in a more concrete way using our case study on maximal ideals in $\ZZ[X]$:

\section{A Case Study: Maximal Ideal in \texorpdfstring{$\ZZ[X]$}{Z[X]}}
\label{Sec:CaseStudy}
For clarity, we will present the classical definition of ideal variants in a ring. Note that in this article, every ring is commutative.
A subset $I\subseteq A$ of a ring $A$ is called an \emph{ideal} if
\begin{itemize}\setlength{\itemsep}{-2pt}
\item[--] $0\in I$,
\item[--] $\forall_{x,y\in I}\ x+y\in I$,
\item[--] $\forall_{\lambda\in A, x\in I}\ \lambda x\in I$.
\end{itemize}
An ideal $P$ is called a \emph{prime ideal} if $1\notin P$ and $\forall_{x,y\in A}(xy\in P\to x\in P \vee y \in I$).
An ideal $M$ is called a \emph{maximal ideal} if $1\notin M$ and there is no ideal $J$ with $M\subseteq J$ and $1\notin J$. Classically (but not constructively) equivalent to the last property is that for all $x\notin M$ there is $\lambda\in A$ with $\lambda x -1 \in M$. 

With these definitions in place, we now illustrate the concept of material interpretation through the case study that every maximal ideal in $\ZZ [X]$ contains a prime number. As is well known, this theorem can even be extended to the statement that every maximal ideal is of the form $$\langle p, f\rangle := \{\lambda p + \mu f\mid \lambda, \mu \in A \},$$ where $p$ is a prime ideal and $f\in \ZZ[X]$ is a polynomial such that its canonical projection $\overline{f}\in (\ZZ/ p\ZZ) [X]$ is irreducible (see \cite[Theorem 5.3]{conradmaximal}). 
However, the main task is to prove the existence of the prime number $p$. The extended statement follows quite directly and is not of particular interest for our purposes.

In the spirit of what we discussed in the first section, particularly Formula~(\ref{Form:StrongNeg}), our aim is to transform the statement 
\begin{quote}
\emph{Let $M\subseteq \ZZ[X]$ be a maximal ideal, then $M$ contains a prime number.}
\end{quote}
along with a classical proof, into a constructive proof of the statement:
\begin{quote}
\emph{Let $M\subseteq \ZZ[X]$ be any subset, then either $M$ contains a prime number or there is concrete evidence that $M$ is not a maximal ideal.}
\end{quote}
We will also work out exactly what is meant by ``concrete evidence that $M$ is not a maximal ideal''.
\subsection{Examination of the Classical Proof}
We begin by presenting the classical proof that every maximal ideal in $\ZZ[X]$ contains a prime number. In doing so, we also present all lemmas whose proofs are necessary to our constructive proof of Theorem \ref{Thm:Main}.  
The statement is often viewed as an application of Zariski's Lemma, which can be used to prove Hilbert's Nullstellensatz \cite{mccabe1976note,zariski1947new}. Indeed, our proof does not directly use Zariski's Lemma but does adopt some of its ideas, as Lemma \ref{Lem:IntField} and Lemma \ref{Lem:IntExt} can also be used in the proof of Zariski's Lemma. We will return to a constructive treatment of Zariski's Lemma in Section \ref{Sec:Zariskis}.
\begin{lemma}\label{Lem:NotField}
$\ZZ [d^{-1}]$ is not a field for $d\in \ZZ\setminus \{0\}$.
\end{lemma}
We present the proof in detail here so that we can later see the parallels in the constructive part. In many articles, this fact -- if mentioned at all -- is proved in only two lines, for example the last two lines in the proof of Theorem~5.1 in \cite{conradmaximal}.
\begin{proof}
Let $q$ be a prime number that does not divide $d$, and assume that $\ZZ[d^{-1}]$ is a field. Then there exist $a_0, \dots, a_n \in \ZZ$ such that $q^{-1} = \sum_{i=0}^n a_i d^{-i}$. Multiplying this equation by $qd^{n}$ gives
$$ d^n = q\sum_{i=0}^n a_id^{n-i},$$
where $\sum_{i=0}^n a_i d^{n-i} \in \ZZ$. Hence, $q$ divides $d^n$, and since $q$ is prime, $q$ divides $d$, which leads to a contradiction.
\end{proof}

For the classical proof, we need the definition of an integral ring extension and some properties of it. For a ring extension $A\subseteq B$, an element $x\in B$ is called \textit{integral} over $A$ if there are $a_0,\dots,a_{n-1}\in A$ with
$$ x^n+a_{n-1}x^{n-1}+\dots+a_1x+a_0 = 0.$$
The ring extension $A\subseteq B$ is called an \textit{integral ring extension} if all $x\in B$ are integral over $A$.

The following two lemmas provide the necessary properties of integral ring extensions we need. 
Even though the proof of particularly the next lemma is straightforward and well known, for example in  \cite[Lemma 2.1.7]{huneke2006integral}. Again, we present it here in full, as the reader will encounter this proof idea in the constructive proof.

\begin{lemma}\label{Lem:IntField}
Let $A\subseteq B$ be an integral ring extension. If $B$ is a field, so is $A$. 
\end{lemma}

\begin{proof}
Let $x\in A\setminus \{0\}$ be arbitrary. As $B$ is a field, there is $x^{-1}\in B$. We show $x^{-1}\in A$:

\noindent The field extension $A\subseteq B$ is integral, hence there are $a_0,\dots,a_{n-1}$ with 
\begin{align*}
(x^{-1})^n+a_{n-1}(x^{-1})^{n-1}+ \dots +a_1x^{-1}+a_0 = 0.
\end{align*}
Multiplying this equation by $x^{n-1}$ and isolating $x^{-1}$ leads to
\begin{align*}
x^{-1}= -a_{n-1}-a_{n-2}x^1-\dots-a_0x^{n-1} \in A.
\end{align*}

\end{proof}

\begin{lemma}\label{Lem:IntExt}
Let $A\subseteq B$ be a ring extension generated by an element $x\in B$, i.e.~$B=A[x]$. If $x$ is integral over $A$, also the ring extension $A\subseteq B$ is integral.
\end{lemma}
The proof is a combination of Proposition 2.4 and Proposition 5.1 in \cite{atiyah2018introduction}:
\begin{proof}
As $x$ is integral over $A$, there are $a_0,\dots,a_{n-1}\in A$ with 
\begin{align}
x^n+a_{n-1}x^{n-1}+\dots+a_1x+a_0 =0. \label{IntEq}
\end{align}
Given $y\in B=A[x]$ there are $b_0,\dots,b_m\in A$ with $y = b_mx^m+\dots+b_1x+b_0$. By Equation (\ref{IntEq}) we can reduce this to 
\begin{align*}
y = c_{n-1}x^{n-1}+\dots+c_1x+c_0
\end{align*}
for some $c_0,\dots,c_{n-1}\in A$. Hence, all elements in $B$ can be represented as a linear combination of $1,x,x^2,\dots, x^{n-1}$ in $A$. Therefore, $B$ is a finitely generated $A$-Module.

With these preliminaries in place, let $z\in B$ be given and we show that $z$ is integral over $A$:

For $i\in \{0,\dots, n-1\}$ we get $a_{ij}\in A$ such that $zx^i = \sum_{j=0}^{n-1} a_{ij}x^j$ by the first part of this proof. Using the Kronecker delta, we obtain
\begin{align*}
\sum_{j=0}^{n-1} (z\delta_{ij}-a_{ij})x^j = 0.
\end{align*}
Let $M$ be the matrix $$M:=\left(z\delta_{ij}-a_{ij}\right)_{ij}.$$ In matrix representation, we obtain
\begin{align*}
M\begin{pmatrix}
1\\ x\\ \vdots \\ x^{n-1}
\end{pmatrix}  = 0.
\end{align*}
Multiplying on the left by the adjoint matrix $\hat{M}$ of $M$ and using $\hat{M}M = \operatorname{det}(M)E$ with $E$ being the unit matrix, we get $\operatorname{det}(M) = 0$ on the first line. Expanding the definition of $M$ and the determinant, we have an integral equation of $z$.
\end{proof}

With this preparation, the proof of the classical statement is now fairly brief:

\begin{theorem}\label{Thm:Class}
Let $M\subseteq \ZZ[X]$ be a maximal ideal. Then there is a prime number $p$ such that $p\in M$.
\end{theorem}
\begin{proof}
If $X\notin M$, there exists some $g\in \ZZ[X]$ with $gX-1\in M$ because $M$ is a maximal ideal. $gX-1$ is not constant since its constant coefficient is $-1$ and $g\neq 0$ because otherwise $1\in M$. Hence, in both cases ($X\in M$ and $X\notin M$) there is some non-constant $f\in M$. Let $d$ be the leading coefficient of $f$.

We now assume that there is no prime number $p$ with $p\in M$. As $M$ is maximal and therefore a prime ideal, it follows $M\cap \ZZ = \{0\}$. Hence, the canonical homomorphism $\ZZ \to \ZZ[X]/M$ is injective and induces a ring extension $\ZZ[d^{-1}]\to \ZZ[X]/M$, where $d^{-1}f$ is an integral polynomial for $X \in \ZZ[X]/M$. As $f\in M$ we have $d^{-1}f = 0$ in $\ZZ[X]/M$. By Lemma \ref{Lem:IntExt} the ring extension $\ZZ[d^{-1}]\to \ZZ[X]/M$ is integral. As $\ZX /M$ is a field, $\ZZ[d^{-1}]$ must also be a field by Lemma \ref{Lem:IntField}, which is a contradiction to Lemma \ref{Lem:NotField}.
\end{proof}

\subsection{Developing a Constructive Proof}
\label{Sec:ConstProof}
What stands out in the classical proof is that it often uses case distinctions based on whether an element belongs to the maximal ideal $M$ or not. If we now want to transform the classical proof into a constructive proof that stays as close as possible to the classical proof, we will need to retain this principle. This means that we must assume that membership in M is decidable.
To make this decidability explicit, we will not view $M$ as a subset of a ring $R$ (in our case $R =\ZZ [X]$) in the constructive proof, but rather as a function from $R$ to the Boolean set $\mathbb{B} := \{\operatorname{tt},\operatorname{ff}\}$. Here, we always consider functions to be total, i.e.~$M(f)$ is either $\operatorname{tt}$ or $\operatorname{ff}$ for all $f\in \ZZ[X]$.
However, we still use the notation $x \in M$ and $x \notin M$, where we mean $M(x) =\operatorname{tt}$ and $M(x) =\operatorname{ff}$, respectively.

We also need to refine the notion of a maximal ideal here, as there are several variants of maximal ideals that are classically equivalent. In the proof, maximality was used in the sense that if an element $a$ is not in $M$, then there is an element $\lambda$ with $\lambda a-1\in M$. 
This leads us to the notion of an explicit maximal ideal, which we even define for any ring $R$. 

\begin{definition}
\label{Def:ExplMaxIdeal}
Let $R$ be a ring. For a subset $M: R \to \mathbb{B}$ and a function $\nu:R\to R$, we say that $(M,\nu)$ is an \textit{explicit maximal ideal} if $M$ is an ideal, $1\notin M$ and $a\nu(a)-1\in M$ for all $a\in R\setminus M$. 

Furthermore, we say that there is \textit{evidence that} $(M,\nu)$ \textit{is not an explicit maximal ideal} if one of the following cases holds:
\renewcommand{\labelenumi}{\theenumi)}
\begin{enumerate}
\item $0\notin M$,
\item there are $a,b\in M$ with $a+b\notin M$,
\item there are $\lambda\in R$ and $a\in M$ with $\lambda a \notin M$,
\item $1\in M$, or
\item there is $a\in R\setminus M$ with $a\nu(a)-1\notin M$.
\end{enumerate}
\end{definition}

\noindent \textit{Remark.} The evidence that $(M,\nu)$ is not an explicit maximal ideal constitutes a stronger version of the statement that $M$ is not a maximal ideal.

Note that if there is evidence that $(M,\nu)$ is not an explicit maximal ideal, we want to be able to explicitly specify which of the statements 1--5 is satisfied. In the case of statements 2, 3, or 5, we also want the elements $a,b$ and $\lambda$ to be explicitly provided.
\vspace{2mm}

The following lemma is the constructive variant of the classical statement that every maximal ideal is also a prime ideal. Classically, this statement is straightforward to prove. Constructively, however, we need to explicitly identify the specific factor of a product in $M$ that also lies in $M$. In short terms: each explicit maximal ideal is an explicit prime ideal.

In our case, however, we even need a stronger form of this statement, which is intended to precisely capture the material interpretation. Therefore, we will formulate this statement and provide a proof as follows:
\begin{lemma}\label{Lem:MaxToPrime}
Let $R$ be a ring, $M: R \to \mathbb{B}$, $\nu: R \to R$ and $a_1,\dots,a_n\in R$ be given such that $a_1\cdot\ldots\cdot a_n \in M$. Then, either some $a_i$ is in $M$, or there is evidence that $(M,\nu)$ is not an explicit maximal ideal.
\end{lemma}
\begin{proof}
Induction over $n$. For $n=0$ it follows that $a_1\cdot\ldots\cdot a_n := 1\in M$, which is evidence that $(M,\nu)$ is not an explicit maximal ideal. For the induction step, let $a_0\cdot\ldots\cdot a_n\in~\!M$. If $a_0\in M$, the proof is complete. Otherwise, $a_0\nu(a_0)-1\in M$ or there is evidence that $(M,\nu)$ is not an explicit maximal ideal.\\
 Moreover , $\nu(a_0)\cdot a_0\cdot\ldots\cdot a_n \in M$ or there is evidence that $(M,\nu)$ is not an explicit maximal ideal. Hence, we can assume that $a_0\nu(a_0)-1\in M$ and $\nu(a_0)\cdot a_0\cdot\ldots\cdot a_n\in M$. This implies $(1- a_0\nu(a_0))\cdot a_1\cdot\ldots\cdot a_n \in M$ and $\nu(a_0)\cdot a_0\cdot\ldots\cdot a_n\in M$ or there is evidence that $(M,\nu)$ is not an explicit maximal ideal. It follows that $a_1\cdot\ldots\cdot a_n\in M$, or there is evidence that $(M,\nu)$ is not an explicit maximal ideal. Applying the induction hypothesis to $a_1\cdot\ldots\cdot a_n\in M$ completes the proof.
\end{proof}

The following lemma provides a more general form of polynomial division, which takes into account that the divisor is not necessarily monic.

\begin{lemma}\label{Lem:PolyDivision}
Let $f,g\in \ZX$ and $d\neq 0$ be the leading coefficient of $f$. Then there is $k\in \NN$ and $h\in \ZX$ such that $\deg(d^kg+hf)<\deg(f)$.
\end{lemma}
\begin{proof}
Let $m:=\deg(f)$ and $n:=\deg(g)$.
For fixed $m$, we use induction on $n$. If $n<m$, we take $k:=0$ and $h := 0$. Otherwise, let $c$ be the leading coefficient of $g$. Then $\deg(dg-cx^{n-m}f) < n$, thus we obtain $k'$ and $h'$ such that $\deg(d^{k'}(dg-cx^{n-m}f)+h'f)<m$. Hence, $k:=k'+1$ and $h:= h'-d^{k'}cx^{n-m}$ fulfil the requirement.
\end{proof}

With these two lemmas, we can now prove the material interpretation of Theorem~\ref{Thm:Class}. Note that we do not need a material interpretation of Lemmas \ref{Lem:NotField}--\ref{Lem:IntExt}, as they are implicitly incorporated into the proof. As a result, the proof itself may appear longer than the classical proof. However, this is not the case, as the corresponding lemmas are already integrated into the proof and even simplified.

\begin{theorem}
\label{Thm:Main}
Let $M\subseteq \ZZ[X]$ and $\nu: \ZZ[X]\to \ZZ[X]$ be an arbitrary map. Then, either there exists a prime number $p\in M$, or there is evidence that $(M,\nu)$ is not an explicit maximal ideal in $\ZZ[X]$.
\end{theorem}
\begin{proof}
First, we construct some non-constant $f\in M$: If $X\in M$, we are done. Otherwise, $X\notin M$, and therefore $X\nu(X)-1\in M$ or there is evidence that $(M,\nu)$ is not an explicit maximal ideal.
Let $d$ be the leading coefficient of $f$ and $n:=\deg(f)$. If $n=0$ we have $-1 = f \in M$ and there is witness that $(M,\nu)$ is not an explicit maximal ideal. Thus, we can assume that $f$ is non-constant and $n>0$.

We take some prime number $q$ which is not a divisor of $d$ and consider $\nu(q)\in \ZX$. We check if $q\in M$ or $m:=q\nu(q)-1\notin M$, if yes, the proof is complete. Otherwise, we continue:
For each $i\in\{0,\dots, n-1\}$ we apply $\nu(q)x^i$ to Lemma \ref{Lem:PolyDivision} and get some $k_i\in\NN$, $h_i\in \ZX$ and $(a_{ij})_{j\in\{0,\dots,n-1\}}\in \ZZ^{n}$ with
\begin{align} \label{Eq:DivRem1}
d^{k_i}\nu(q)x^i + h_if = \sum_{j=0}^{n-1}a_{ij}x^j.
\end{align}
Using the Kronecker delta $(\delta_{ij})_{ij}$ we get
\begin{align}\label{Eq:DivRem2}
\sum_{j=0}^{n-1}(d^{k_i}\nu(q)\delta_{ij}-a_{ij})x^j = -h_if.
\end{align}
Let $A$ be the matrix $(d^{k_i}\nu(q)\delta_{ij}-a_{ij})_{i,j\in\{0,\dots,n-1\}}$ then we have 
\begin{align}\label{Eq:MatrixNota1}
A \begin{pmatrix}
1\\ x\\ \vdots \\ x^{n-1}
\end{pmatrix} = \begin{pmatrix}
-h_0f\\ -h_1f\\ \vdots \\ -h_{n-1}f
\end{pmatrix}
\end{align}
Multiplying both sides by the adjugate matrix $\hat{A}$ of $A$ and using $\hat{A}A = \det(A)I$ leads to
\begin{align}\label{Eq:MatrixNota2}
\begin{pmatrix}
\det(A)\\ \det(A)x\\ \vdots \\ \det(A)x^{n-1}
\end{pmatrix} = \hat{A}\begin{pmatrix}
-h_0f\\ -h_1f\\ \vdots \\ -h_{n-1}f
\end{pmatrix}
\end{align}
In particular, the first line is $\det(A) = -\sum_{j=0}^{n-1}\hat{A}_{0j}h_jf$. Looking at the definition of $A$, we have $\det(A) = d^K\nu(q)^n+b_{n-1}\nu(q)^{n-1}+\dots+b_1\nu(q)+b_0$ for some $b_0,\dots,b_{n-1}\in \ZZ$ and $K:= \sum k_i$. Hence, 
\begin{align}
d^K\nu(q)^n+b_{n-1}\nu(q)^{n-1}+\dots+b_1\nu(q)+b_0 = \sum_{j=0}^{n-1}(-\hat{A}_{0j}h_j)f.
\end{align}
Multiplying both sides by $q^n$ leads to 
\begin{align}
d^K(q\nu(q))^n+b_{n-1}q(q\nu(q))^{n-1}+\dots+b_1q^{n-1}(q\nu(q))+b_0q^n = \sum_{j=0}^{n-1}(-q^n\hat{A}_{0j}h_j)f
\end{align}
We define $m:= q\nu(q)-1$, which is equivalent to $q\nu(q)= m+1$. For each $i\in \{1,\dots,n\}$ one can easily compute some polynomial $g_i$ with $(m+1)^i = 1+mg_i$. This leads to
\begin{align}
&d^K+b_{n-1}q+\dots+b_1q^{n-1}+b_0q^n\notag =\\
& \sum_{j=0}^{n-1}(-q^n\hat{A}_{0j}h_j)f + (-d^Kg_n-b_{n-1}qg_{n-1}-\dots-b_1q^{n-1}g_1)m
\label{Eq:End}
\end{align}
Since the first part of the equation lies in $\ZZ$, the second part must also lie in $\ZZ$. Furthermore, the first part cannot be zero because otherwise $q\mid d$ (or $q\mid 1$ if $K=0$) and the second part is a linear combination of the elements $f$ und $m$, which are both in $M$. By Lemma \ref{Lem:MaxToPrime}, one prime factor is in $M$ or there is evidence that $(M,\nu)$ is not an explicit maximal ideal.
\end{proof}
\subsection{Translation from the Classical to the Constructive Proof}
\label{Sec:Transl}
First, we would like to review and summarize how we come from the classical statement to the constructive one.
As stated in in Section \ref{Sec:Intro}, we wanted to translate a classical proof of $A\to B$ into a classical proof of $\Neg \tilde{A} \vee \tilde{B}$. Here, $A$ was the statement ``$M$ is maximal ideal'' and $B$ was the statement ``$M$ contains a prime number''.

When introducing classical maximal ideals, we presented two variants. As a result, the statement ``A is a maximal ideal'' is, in fact, a classical statement. Constructively, this statement must be made more specific. Already the first sentence of the classical proof  makes clear which version applies: Here, we concluded from $X\notin M$ that there exists a $g$ such that $gX-1\in M$. For this reason we had to use the classical stronger version of maximality, namely for all $f \notin M$ there is $g$ with $gf-1\in M$. Definition \ref{Def:ExplMaxIdeal} is then the strong negation of this version.

A similar situation occurred with the type of $M$: in the classical proof, $M$ was a subset, but the proof sometimes required knowing whether a given element was in $M$ or not.  Therefore, we replaced the vague notion of a set with the precise statement that $M$ is a total function into $\mathbb{B}$.
Moreover, we also made frequent use of this property in the constructive proof, which now allows us to discuss how we transformed the classical proof into a constructive one:

What immediately stands out is that, in the constructive proof, instances of modus ponens from the classical proof are often replaced by a case distinction. This is particularly the case when statements of the form $a\in M$ are considered. For example, a statement such as ``$a\in M$ because $M$ is a maximal ideal'' is replaced by the statement ``$a\in M$ or we have evidence that $M$ is not an explicitly maximal ideal''. 
This is another reason why the assumption that $a\in M$ is decidable was indeed necessary.

Apart from this transformation, the first part of the proof, where the non-constant $f\in M$ is obtained, is quite similar. 
However, in the classical proof, this $f$ is only used as a witness that $\ZZ[d^{-1}] \to \ZZ[X]/M$ is an integral ring extension. In the constructive proof, on the other hand, the notion of an integral ring extension and the embedding $\ZZ[d^{-1}] \to \ZZ[X]/M$  do not explicitly appear but are incorporated into the proof itself.
For this reason, $f$ is used further; in particular, it is applied to Lemma \ref{Lem:PolyDivision}, and appears in several subsequent equations, such as Equations (\ref{Eq:DivRem1}) to (\ref{Eq:End}). The same holds for the prime number $q$, which is coprime to  the leading coefficient $d$ of $f$. The classical proof only uses it in Lemma \ref{Lem:NotField}, whereas the constructive proof defines it directly after the construction of $f$ and then uses $q$ and $\nu(q)$ throughout the proof. 
One can see that the details of the proof become much clearer in the constructive proof, whereas in the classical proof, they are less explicit.

Note also that the constructive proof consistently argues within $\ZZ[X]$ or $\ZZ$. As $\ZZ$ and $\ZZ[X]$ are concrete objects, we do not need to make case distinctions but can argue directly. Which we did to use statements such as $\hat{A}A = \det(A)I$ and ``the first part can not be zero as otherwise $q\mid d$ (or $q\mid 1$ if $K=0$)''.
\subsection{The Constructive Content}
\label{Sec:ConstrCont}
An Python implementation of the algorithm, derived from the constructive proof in Section \ref{Sec:ConstProof}, is provided at the author's GitHub page (\url{https://github.com/FranziskusWiesnet/maxzx}) using the built-in computer algebra system \textit{SymPy} \cite{sympy2024}. The main algorithm is referred to as \texttt{MaxZX}. In addition, there is a simple algorithm \texttt{unbounded\_search} that takes any $M:\ZX \to \mathbb{B}$ and performs an unbounded search over all prime numbers and tests which prime lies in $M$.
There are two examples: one is the ideal generated by $X^2+1$ and $3$, and the other is the ideal generated by $X$ and $1019$.

While the unbounded search finds a prime (in this case, 3) in $M$ significantly faster than \texttt{MaxZX} in the first example, the situation is clearly the opposite in the second case. The larger the prime number and the higher the runtime of $M$, the longer the unbounded search takes. In contrast, the runtime of \texttt{MaxZX} scales with the runtime of $M$ and $\nu$.

\texttt{MaxZX} also has the advantage that it allows the use of arbitrary subsets and functions. Moreover, it provides a witness that explains why the conditions are not satisfied. For example, if we pass the constant false function $(\lambda_{-} \operatorname{False})$ for $M$ and the zero function $(\lambda_- 1)$ for $\nu$ to \texttt{MaxZX}, it returns $(\operatorname{False}, 3, x)$. This indicates that $(M, \nu)$ is not an explicit maximal ideal, as witnessed by $X\nu(X) - 1 \notin M$.

One can now proceed, for example, by setting $\nu(x) = 2$ and $M(2x - 1) = \text{true}$. Then \texttt{MaxZX} will produce a different witness showing why $(M, \nu)$ is not a maximal ideal. Then one can further modify $M$ and $\nu$ so that this case also does not lead to a counterexample, and \texttt{MaxZX} outputs a different counterexample or a prime number.
 This process can be continued as far as one wishes. This is illustrated in the Python file with the example \texttt{M\_2} and \texttt{nu\_2} in the Python code.
In \cite{powell2019algorithmic,powell2022universal}, we followed a similar procedure in the context of general maximal objects. This demonstrates that even counterexamples or failure evidence can be informative in constructive mathematics.

\section{Outlook}
\subsection{Further Applications of Material Interpretation}
The material interpretation offers a versatile framework with potential applications in various areas of mathematics and logic.
The example above served as a case study demonstrating the underlying idea. There are many cases where material interpretation can be of substantial benefit. The aim of future research will be to explore these examples further. 
We highlight the following two applications, where some preliminary work has already been carried out:
\subsubsection{Zariski's Lemma and Hilbert's Nullstellensatz.}
\label{Sec:Zariskis}
Zariski’s Lemma can be stated as follows:
\begin{quote}
Let $K$ be a field and $R$ a $K$-algebra, which is also a field. Suppose that $R = K[x_1,\dots,x_n]$ for some $x_1,\dots,x_n\in R$. Then $R$ is algebraic over $K$, i.e.,~there are non-zero $f_1,\dots,f_n\in K[X]$ such that $f_i(x_i)=0$ for all $i$.
\end{quote} 
An algorithmic version of Zariski's Lemma is considered in \cite{wiesnet2021algorithmic}. 
There, a constructive proof of Zariski's Lemma is provided, from which an algorithm is then developed, and the correctness of the algorithm is proved. 
The algorithmic version is divided into several lemmas, but they all share the structure in which an algorithm is first presented, followed by a theorem about the algorithm. These theorems typically state that if the objects in the algorithm satisfy certain axioms (e.g.~field or ring axioms), then the algorithm’s output satisfies certain properties. Thus, a material interpretation means that either the output satisfies the desired properties, or there exists a witness showing that one of the axioms about the objects in the algorithm does not hold.
A proof sketch was already provided in Section 4.5.1 of \cite{wiesnet2021computational} and will be the subject of further research.
In Section 4.5.2 of \cite{wiesnet2021computational}, it was shown that applying the material interpretation to Zariski’s Lemma yields an algorithmic version of Hilbert’s Nullstellensatz, where Hilbert's Nullstellensatz states the following:
\begin{quote}
Let $K$ be an algebraically closed field, and let $f_1,\dots,f_m\in K[X_1,\dots,X_n]$ be given. Then either there are $g_1,\dots,g_m \in K[X_1,\dots,X_n]$ with $g_1f_1+\cdots +g_mf_m =1$ or there are $x_1,\dots,x_n\in K$ with $f_i(x_1,\dots,x_n) = 0$ for all $i$.
\end{quote}

What remains is to fully prove the material interpretation of Zariski's Lemma and to isolate its constructive content, preferably as a computer program.
As previously mentioned, the case study in Section \ref{Sec:CaseStudy} of this article is a special case of Zariski's Lemma. More precisely, it follows from Zariski's Lemma that, considering the embedding $\QQ \to \ZZ[X]/M$, the element $X$ is algebraic over $\QQ$, i.e.~ there is some $g\in \QQ[X]$ with $g(X) = 0 \in \ZZ[X]/M$. 
If we now multiply $g$ by the least common denominator of all its coefficients, we indeed obtain $f$, which is constructed in the first part of the proof. Thus, the first part of the proof was essentially the base case of the inductive proof of Zariski’s Lemma. 
Interestingly, Lemma \ref{Lem:IntField} and Lemma \ref{Lem:IntExt} are also applied in the induction step of Zariski's Lemma \cite[Lemma 1--2]{wiesnet2021algorithmic}.
Therefore, studying it is particularly well-suited as a stepping stone to understanding the material interpretation of Zariski's Lemma, and consequently, of Hilbert's Nullstellensatz.
\subsubsection{The Universal Krull-Lindenbaum Lemma.}
In \cite{rinaldi2016universal}, Schuster and Rinaldi gave a constructive version of the Universal Krull-Lindenbaum lemma (UKLL), and in \cite{powell2019algorithmic,powell2022universal} an algorithm realising UKLL under some general assumptions is developed. 

\noindent \textbf{Ring theory.} In the context of ring theory the universal Krull-Lindenbaum lemma says that the intersection of all prime ideals of a commutative ring $R\neq \{0\}$ is equal to the nilradical:
\begin{align*}
\bigcap \left\{P\subseteq R\mid P\text{ is a prime ideal in }R\right\} = \sqrt{\langle 0\rangle}
\end{align*}

For a fixed $r\in R$, the construed algorithm, takes a so-called \textit{Krull functional} as input. In the context of ring theory a Krull functional for a subset takes an arbitrary subset $P\subseteq R$ and returns a witness that either $r\in P$ or that $P$ is not a prime ideal. This can be seen as a material interpretation of the statement ``r is in every prime ideal''.
The output of the constructed algorithm in the context of ring theory is indeed an $e\in \mathbb{N}$ such that $r^e = 0$.
Through the material interpretation of general proofs of the following statements, which use UKLL  in the version above, we were indeed able to give constructive variants of the following theorems:
\begin{quote}
Let $A$ be a ring and $f = \sum_{i=0}^n a_iX^i \in A[X]$ be a unit in $A[X]$, then $a_i$ is nilpotent for each $i>0$.
\end{quote}

\begin{quote}
Let $A$ be a ring and $f = \sum_{i=0}^n a_iX^i,g = \sum_{i=0}^m b_iX^i \in A[X]$ be two polynomials with $fg = \sum_{i=0}^{n+m} c_iX^i$, then
\begin{align*}
a_ib_j \in \sqrt{\langle c_0,\dots,c_{i+j} \rangle}\ .
\end{align*}
\end{quote}
The first statement is a typical example and has already been studied, among others, in \cite{persson1999application,schuster2013induction}. The second statement is known as theorem by Gauß-Joyal and its constructive meaning was also studied as a standard example in many articles such as \cite{banaschewski1996polynomials,coquand2009space}.

Hence, in the case of ring theory and on test examples, the material interpretation has already shown success, partly due to the simplicity of the test examples. In the context of ordered fields, however, the proofs are significantly more complex:

\noindent\textbf{Ordered Fields.}
For a field $K$ a total order $\leq$ on $K$ is called a \textit{field order}, if for all $a,b,c\in K$: $a\leq b\rightarrow a+c \leq b +c$ and $0\leq a \wedge 0\leq b \rightarrow 0\leq a\cdot b$. We call a subset $U\subseteq K$ an \textit{order} if there is a field order $\leq$ on $K$ such that $U = \{ x\in K \mid 0 \leq x\}$.
In this setting the universal Krull-Lindenbaum-Lemma coincides with the Artin-Schreier Theorem and says
\begin{align*}
\bigcap \left\{U\subseteq K\mid U\text{ is an order of }K\right\} = \left\{\sum_{i=0}^n x_i^2 \ \middle|\ n\in \NN,\ x_0,\dots,x_n\in K\right\}.
\end{align*}
This statement is only non-trivial if $-1$ is not a sum of squares in $K$ which implies $\operatorname{char}(K) = 0$. The fields $\RR$ and $\QQ$ and all function fields over them fulfil this property.

Artin proved the Artin-Schreier Theorem and used it to prove Hilbert's 17th Problem \cite{artin1927ueber}:

\begin{quote}
Let $f\in \QQ[X_1,\dots,X_n]$ be a polynomial in $n\in\NN$ variables such that $f(x)\geq 0$ for all $x\in \mathbb{Q}^n$. Then $f$ is the sum of squares in $\mathbb{Q}(X_1,\dots,X_n)$.
\end{quote}

\noindent Artin's proof of Hilbert's 17th problem, however, is very complex and involves, in addition to the Artin-Schreier theorem, Sturm's theorem \cite[Section 5.8.2]{freund2007stoer} and the fact that every ordered field has a real closure. A constructive consideration of the real closure was given by Lombardi and Roy \cite{lombardi1991elementary} and could be used to construct a Krull functional.

There is already a constructive proof of Hilbert's 17th problem by Delzell \cite{delzell1984continuous}, and it would be very interesting to compare the computational content of Delzell's proof with that of the proof using a Krull functional.
However, since Artin's proof is very complex, it will hardly be possible to trace the computational content using only pen and paper. Arguably, the same applies to the material interpretation of Zariski's Lemma for the section above, especially when aiming to extract its precise constructive content. Therefore, computer assistance will be absolutely necessary.
\subsection{Implementation of the Material Interpretation}
In Section \ref{Sec:ConstrCont}, we discussed an implementation of the algorithm in Python. The algorithm was created manually. However, it is of interest to explore whether this algorithm could be generated automatically. This applies not only to the algorithm itself, but also to the entire material interpretation as well. An implementation in a proof assistant such as Agda, Coq, or Lean would be highly valuable.

The ultimate goal would be to at least partially automate the material interpretation.
As we have seen in Section \ref{Sec:Transl}, applications of modus ponens are often replaced by a case distinction. Such transformations could be implemented in a proof assistant.
Furthermore, in the classical proof, a case distinction based on $x\in M$ was performed and the property that 
$M$ is a maximal ideal was utilized to imply that for $x\notin M$, there exists an $a$ with $ax-1\in M$.
In the constructive setting, this motivates defining $M$ as a total Boolean-valued function and the definition of an explicit maximal ideal. 
Thus, it can be observed that the modified statement can, in fact, be extracted from the classical proof. This suggests that automation might be possible here as well.

\bibliographystyle{plain}
\bibliography{bib_maxzx}

\end{document}